\documentclass[final,1p,times]{elsarticle}
\usepackage{amssymb}
\usepackage{amsmath}
\usepackage{enumitem}
\usepackage{algorithm}
\usepackage{appendix}
\usepackage{algorithmic}
\usepackage{graphicx}
\usepackage{tabularx}
\usepackage{booktabs}   
\usepackage{gensymb}
\usepackage{subcaption}
\usepackage{hyperref}
\usepackage{epsfig}
\usepackage{subcaption}
\usepackage{amsthm}
\usepackage{lmodern}
\usepackage{xcolor}

\newtheorem{theorem}{Theorem}[section]
\newtheorem{corollary}[theorem]{Corollary}

\newtheorem{proposition}[theorem]{Proposition}
\newtheorem{definition}[theorem]{Definition}
\newtheorem{remark}[theorem]{Remark}

\DeclareMathOperator{\divop}{div}

\hypersetup{
	colorlinks=true
}
\usepackage{natbib} 
\makeatletter
\def\ps@pprintTitle{%
	\let\@oddhead\@empty
	\let\@evenhead\@empty
	\def\@oddfoot{}%
	\let\@evenfoot\@oddfoot}
\makeatother

\begin{document}

\begin{frontmatter}

\title{Gauge-energy preservation under congestion-controlled network repair} 
\author{{Zhenyuan Sun}}
\ead{2300010715@stu.pku.edu.cn}
\author{Dayue Chen}

\address{School of Mathematical Sciences, Peking University, Beijing 100871, China.}

\begin{abstract}
We study preservation of finite gauge-energy flows under local network repair after Bernoulli
edge failures. A macroscopic demand network records terminal pairs to be routed, while a
microscopic physical network contains local backup routes, bypasses and shared corridors. We
prove a deterministic gauge-energy repair theorem: if the usable demand network $\mathcal{B}^{\sharp}$ carries a finite $\Phi$-energy flow $\theta$, then the repaired physical network $H$ carries a lifted finite
$\Phi$-energy flow $\Theta$ with $$
\mathcal{E}^\Phi_H(\Theta) \le L\,\beta_\Phi(K)\,\mathcal{E}^\Phi_{\mathcal{B}^{\sharp}}(\theta),
$$ where $L$ bounds route length, $K$ bounds routing congestion, $\mathcal{E}$ marks energy, and $\beta_\Phi$ is the gauge dilation constant. We then convert this comparison into probabilistic repair criteria:
finite-dependent local repair is handled via domination by product measures, and random repair
lengths via a variable-cost formulation compatible with chemical-distance estimates. As a main
application, we prove a finite-dependent local bypass theorem: any macroscopic network whose
supercritical percolation cluster supports a finite gauge-energy flow remains gauge-energy stable
after bounded-range local reinforcement, provided the local repair probability is sufficiently
high. This yields reinforced lattice and wedge-type examples and provides a potential-theoretic
framework for random network repair beyond tree-like or edge-disjoint constructions.

\end{abstract}


\end{frontmatter}

\section{Introduction}
\textbf{Background and motivation. }
Random perturbations of infinite graphs provide a basic setting for studying how local changes in
network structure affect global analytic and flow properties~\cite{Grimmett,Li2021}. A natural
question is which properties of an infinite graph remain stable when its local connectivity is
randomly altered. Bernoulli bond percolation provides a standard model for such perturbations:
each edge of a graph is independently retained with probability $p$ and deleted otherwise
~\cite{Grimmett,LP}. The resulting random subgraph leads to questions concerning connectivity,
random walks, effective resistance, and finite-energy flows~\cite{DS,LP}. In particular,
understanding how local edge failures influence global flow properties is an important problem
in the potential theory of random networks and repairing such failures is also an interesting topic.

Among the analytic properties encoded by finite-energy flows, transience is of particular
interest. By the electrical-network criterion, transience is equivalent to the existence of a
unit flow to infinity with finite quadratic energy given by the functional
\begin{equation*}
	\mathcal{E}(\theta)=\sum_e R_e\,\theta(e)^2<\infty,
\end{equation*}
where $R_e$ denotes the resistance of edge $e$ and $\theta(e)$ denotes the flow through edge
$e$~\cite{DS,LP}. Rayleigh monotonicity implies that recurrence is preserved under edge
deletion, whereas the preservation of transience is more delicate because removing edges may
destroy the routes supporting finite-energy flows to infinity. Thus, network perturbations lead
naturally to the question of whether finite-energy flows can be preserved after local changes of
the network. 

Beyond the quadratic setting, this viewpoint extends to more general convex gauges
$\Phi$ and finite $\Phi$-energy flows satisfying~\cite{HM,LP2}
\begin{equation*}
	\mathcal{E}^{\Phi}(\theta)=\sum_e \Phi(|\theta(e)|)<\infty.
\end{equation*}
This broader formulation provides a natural framework for studying the stability of energy-controlled flows under
general network modifications, including random failures and local repair mechanisms.

The preservation of finite-energy flows under random perturbations has been studied extensively
in electrical networks and percolation theory. For a graph $G=(V,E)$, Bernoulli edge percolation
produces a random subgraph $G_p=(V,E_p)$ with $E_p\subseteq E$, and a central question is
whether analytic properties of $G$ are preserved on $G_p$. 
The electrical-network framework of
Doyle and Snell~\cite{DS} and the subsequent development of Lyons and Peres~\cite{LP}
established the role of unit flows and energy minimization in characterizing such properties.  
Studies on percolation perturbations have further investigated
how potential-theoretic and random-walk properties behave on percolation clusters
~\cite{BLS}.
Classical percolation results, including those of Grimmett, Kesten and Zhang~\cite{GKZ},
established transience properties of infinite percolation clusters. Further analytic properties
of random walks on supercritical clusters, such as heat kernel estimates, were developed by
Barlow~\cite{Barlow2004}. Beyond the quadratic setting, finite-energy flows defined through
general convex gauges have also been studied. Hoffman and Mossel~\cite{HM} investigated the
stability of finite gauge-energy flows under thinning of $\mathbb{Z}^d$, while Levin and Peres
developed criteria based on convex energy gauges and cutsets~\cite{LP2}.\\

The above results mainly describe random perturbations in which the modified network is obtained by retaining or deleting edges of the original graph. In such settings, the flow is studied directly on the surviving subgraph. However, many network models with local redundancy preserve connectivity after failures through bypasses, redundant connections, or other repair mechanisms. Related ideas have been explored in complex and infrastructure networks, where bypass rewiring and post-failure repair strategies are employed to maintain or restore connectivity after disruptions~\cite{ParkHahn2016,Chujyo,Canbilen}. Flow-based analyses have further investigated how structural degradation and failure processes affect network-level flow behavior~\cite{Hamedmoghadam}.

Unlike simple edge deletion, repair mechanisms generate new random edge configurations with nontrivial dependence structures. Finite-dependent and positively correlated percolation models have received renewed attention in this context. Liggett, Schonmann and Stacey~\cite{LSS} established the canonical domination-by-product-measures criterion, Temmel~\cite{Temmel2014} extended it to non-homogeneous marginals, and Köhler-Schindler and Sulser ~\cite{KS24} further developed domination results under positive association assumptions. 
In repaired networks, a failed edge does not necessarily eliminate the corresponding demand; instead, the demand may be realized through newly constructed routes. Therefore, preserving a finite-energy flow in such networks requires quantitative control of the flow transfer through repair routes and the additional energy introduced by the repair process. These observations motivate the study of gauge-energy preservation under congestion-controlled local network repair.\\

\textbf{Main results.} We introduce a two-level network repair setting to study the preservation
of finite $\Phi$-energy flows after random failures. A macroscopic demand network
$\mathcal{B}=(T,E_{\mathcal{B}})$ with a distinguished root $o$ records terminal connections that need to be maintained, while a
microscopic physical network $H=(V_H,E_H)$ provides local repair routes for realizing these
connections through an injective terminal map $\iota:T\to V_H$. Instead of requiring a demand edge
to survive as a single physical edge, we allow it to remain usable whenever its terminal pair can be
connected through an available local repair route. Note that \(\iota\) is not required to be the identity: macroscopic demand vertices may be embedded at distinct physical locations, and a demand edge may be realized by a local bypass rather than a single physical edge. The special case \(\iota = \mathrm{id}\) recovers ordinary Bernoulli percolation on \(H\), in which a demand edge survives precisely when its corresponding physical edge is open; the general case \(\iota \neq \mathrm{id}\) is the setting of interest, where redundancy is provided by local repair routes rather than by direct edge survival.

The first main result is a deterministic flow-lifting theorem (Theorem~\ref{thm:main}). Suppose
that the usable demand network $\mathcal{B}^{\sharp}=(T,F)$ is repaired in $H$ by an $(L,K)$-controlled scheme,
where every demand edge is implemented by a physical route of length at most $L$, and at most $K$
routes share any physical edge. If $\mathcal{B}^{\sharp}$ supports a unit flow $\theta$ with finite $\Phi$-energy,
then the component of $\iota(o)$ in $H$ supports a lifted unit flow $\Theta$ satisfying
\begin{equation}\label{eq:intro-main-xu}
	\mathcal{E}^\Phi_H(\Theta)\le L\,\beta_\Phi(K)\,\mathcal{E}^\Phi_{\mathcal{B}^{\sharp}}(\theta).
\end{equation}
Here, $L$ measures the cost of route length and $K$ controls the congestion caused by route overlap.

The second main result gives a probabilistic repair sufficient condition (Theorem~\ref{thm:fd}). For
bounded-range local repair models under Bernoulli failures, we show that suitable domination
conditions on the repairability field imply the preservation of finite $\Phi$-energy flows with
positive probability. A random-cost extension (Theorem~\ref{thm:random}) further allows repair
routes with random lengths.

Finally, we apply the setting to lattice-based local repair models and prove a finite-dependent
local bypass theorem (Proposition~\ref{prop:sq}). These examples show how local redundancy, route
length, and congestion determine the cost of preserving macroscopic gauge-energy flows after random
failures.
The proposed approach provides sufficient conditions for preserving finite gauge-energy flows after
local repair. It is not intended as a characterization of all random networks or all possible
repair mechanisms, but rather applies when a suitable demand structure and controlled local
repair routes are available.

\textbf{Proof strategy.} The proof combines a deterministic flow lifting argument with a
probabilistic repair sufficient condition. The deterministic part constructs a lifted flow along selected
repair routes and controls the additional $\Phi$-energy through route length and congestion. The
probabilistic part uses domination arguments to guarantee the availability of suitable repair routes
under random failures.

The paper is organized as follows. Section~\ref{sec:prelim} introduces demand networks, physical
networks, flows and gauge energies. Section~\ref{sec:det} proves the deterministic repair theorems,
including the weighted variable-cost version. Section~\ref{sec:prob} records probabilistic repair
criteria and the external percolation inputs used later. Section~\ref{sec:app} presents the
applications, including independent corridor reinforcement, finite-dependent square-bypass networks,
lattice and wedge-type consequences, coarse block repairs and finite-scale resistance comparisons. 

\section{Networks, flows and gauge energies}\label{sec:prelim}

We first introduce the network structures, flow notions, and energy functionals used throughout the paper.
The model distinguishes between a macroscopic demand network, which describes the connections to be maintained, and a microscopic physical network, which provides possible repair routes after failures. 
Then, we define flows on these networks and the gauge energies used to measure their costs.
In this paper, graphs are assumed to be locally finite and undirected, with flows represented on oriented edges. And when considering flows to infinity, we work on infinite graphs. Parallel edges are allowed in the macroscopic demand network.

\begin{definition}[Demand and physical networks]\label{def:networks}
	A \emph{demand network} is a connected locally finite multigraph
	\begin{equation*}
		\mathcal{B} = (T, E_{\mathcal{B}})
	\end{equation*}	where \(T\) is the set of terminals and \(E_{\mathcal B}\) is the set of demand edges, with a distinguished root $o \in T$. A \emph{physical network} is a locally finite graph
	\begin{equation*}
		H = (V_H, E_H)
	\end{equation*}
	where \(V_H\) and \(E_H\) are the vertex and edge sets, together with an injective terminal map $\iota: T \to V_H$. A demand edge $e = \{u, v\}
	\in E_{\mathcal{B}}$ represents a macroscopic connection between the physical terminals
	$\iota(u)$ and $\iota(v)$. 
\end{definition}

We recall the standard notion of flows on infinite networks; see  
\cite{DS, LP}.
\begin{definition}[Unit flow]\label{def:flow}
	A \emph{unit flow from $o$ to infinity} on a locally finite graph is an antisymmetric
	function $\theta$ on oriented edges satisfying
	\begin{equation*}
		\divop \theta(o) = 1, \qquad \divop \theta(x) = 0 \quad (x \ne o),
	\end{equation*}
	where
	\begin{equation*}
		\divop \theta(x) = \sum_{y:\, y \sim x} \theta(x, y)
	\end{equation*}
	denotes the net flow leaving $x$. Here $\theta(x, y)$ denotes the flow along the oriented edge $(x,y)$. For a finite terminal set $A$, $o \notin A$, a \emph{unit flow from the root $o$ to $A$} is defined similarly, with
	$\divop \theta(o) = 1$, with $\divop \theta(x) = 0$ for $x \notin \{o\} \cup A$, and with total
	divergence $\sum_{x \in A} \divop \theta(x) = -1$ on $A$.
\end{definition}

For unit resistances, the quadratic energy is
\begin{equation*}
	\mathcal{E}(\theta) = \sum_{e\in E} \theta(e)^2 .
\end{equation*}

By Thomson’s principle and the electrical-network criterion, an infinite locally finite graph is transient if and only if it supports a unit flow to infinity with finite quadratic energy~\cite{DS,ABBP,LP}.

To measure flow costs beyond the classical quadratic energy, we use the gauge-energy framework developed for flows on percolation clusters~\cite{HM,LP2}. The following definition introduces the corresponding energy functional.
\begin{definition}[Gauge energy]\label{def:gauge}
	A \emph{gauge} is an increasing convex function $\Phi: [0, \infty) \to [0, \infty)$ with
	$\Phi(0) = 0$. The $\Phi$-energy of a flow $\theta$ on a graph $X$ is defined as
	\begin{equation*}
		\mathcal{E}^\Phi_X(\theta) = \sum_{e \in E_X} \Phi(|\theta(e)|).
	\end{equation*}
	For $\Phi(t) = t^q$, $q > 1$, we write
	\begin{equation*}
		\mathcal{E}^q_X(\theta) = \sum_{e \in E_X} |\theta(e)|^q .
	\end{equation*}
\end{definition}

To quantify the energy increase caused by combining multiple flows on a shared edge, for an integer \(m\geq1\) representing the number of flows, we define the dilation constants

\begin{equation}\label{eq:dilation}
	\alpha_\Phi(m) = \sup_{t > 0} \frac{\Phi(mt)}{\Phi(t)}, \qquad
	\beta_\Phi(m) = \frac{\alpha_\Phi(m)}{m}, \qquad
	\beta_\Phi(K) = \max_{1 \le m \le K} \beta_\Phi(m) .
\end{equation}
These quantities describe the growth of the gauge when \(m\) flow contributions are merged. The arguments below only require these quantities to be finite for the congestion levels under consideration. For the power gauge \(\Phi(t)=t^q\),

\begin{equation}\label{eq:power}
	\alpha_\Phi(m) = m^q, \qquad \beta_\Phi(m) = m^{q-1}, \qquad \beta_\Phi(K) = K^{q-1}.
\end{equation}

The following convexity estimate will be used throughout the paper.   For \(m\geq1\) and \(x_1,\ldots,x_m\geq0\), let $\bar{x} = m^{-1}\sum_{i=1}^m x_i$. Since $\Phi$ is convex,
$\Phi(\bar{x}) \le m^{-1}\sum_{i=1}^m \Phi(x_i)$,   Jensen's inequality, and  the definition
of $\alpha_\Phi$ give
\begin{equation}\label{eq:jensen}
	\Phi\bigl(m\bar{x}\bigr) \le \alpha_\Phi(m) \Phi(\bar{x})
	\le \beta_\Phi(m) \sum_{i=1}^m \Phi(x_i).
\end{equation}
This is the exact convex-energy cost of merging $m$ flows on a shared edge: the factor
$\beta_\Phi(m)$ is the average dilation of $\Phi$ over $m$ arguments.

\section{Deterministic repair theorems}\label{sec:det}

We now introduce the deterministic repair model used to transfer flows from a demand network to a physical network. The main idea is to replace failed demand edges by selected repair paths in the physical network. The energy cost of this transfer is controlled by two quantities including the maximum route length and the congestion caused by shared physical edges.

The following definition formalizes a repair scheme.

\begin{definition}[Repair scheme]\label{def:scheme}
	Let $F \subseteq E_{\mathcal{B}}$. A \emph{repair scheme} for $F$ assigns to each demand edge
	$e = \{u, v\} \in F$ a finite path
	\begin{equation*}
		\gamma_e \subseteq H
	\end{equation*}
	joining $\iota(u)$ to $\iota(v)$. The usable demand network is $\mathcal{B}^{\sharp} = (T, F)$.
\end{definition}

For a physical edge $a \in E_H$, let
\begin{equation*}
	I(a) = \{e \in F : a \in \gamma_e\}, \qquad m_a = |I(a)|.
\end{equation*}
The number $m_a$ is the \emph{congestion} of the selected repair scheme at $a$.

The repair cost depends on two quantities, the length of repair paths and the congestion on physical edges. Therefore, we introduce the following notion $(L,K)$ -controlled repair.
\begin{definition}[$(L,K)$-controlled repair]\label{def:controlled}
	A repair scheme is \emph{$(L,K)$-controlled} if
	\begin{equation*}
		\sup_{e \in F} |\gamma_e| \le L, \qquad \sup_{a \in E_H} m_a \le K.
	\end{equation*}
\end{definition}

The $(L,K)$-controlled condition bounds two key features of a repair scheme, namely the length of repair paths and the congestion of shared physical edges. The following theorem shows that, under this condition, a finite 
$\Phi$-energy flow on the demand network can be transferred to the physical network with a controlled energy cost. Meanwhile, figure \ref{fig:fig01} shows the schematic diagram of the proof process for Theorem \ref{thm:main}. 

\begin{figure}[htbp]
	\centering
	\includegraphics[width=0.6\textwidth]{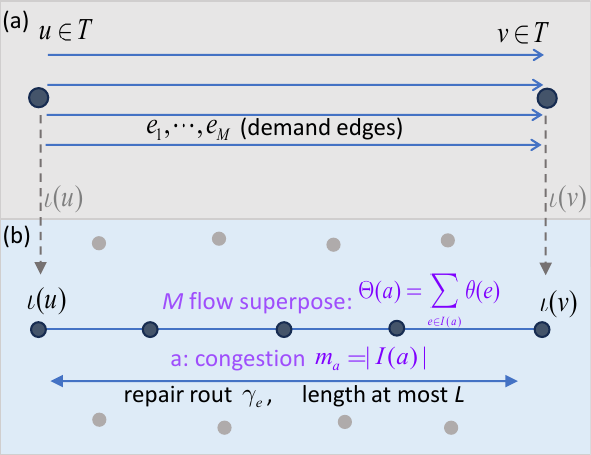}

\caption{Two-layer repair model: (a) Demand network
	$\mathcal B=(T,E_{\mathcal B})$ with terminals $u,v\in T$ and $M$ parallel
	demand edges $e_1,\dots,e_M$; (b) Physical network $H=(V_H,E_H)$ with an
	injective terminal map $\iota:T\to V_H$. A demand edge $e=\{u,v\}$ is
	replaced by a repair route $\gamma_e$ from $\iota(u)$ to $\iota(v)$; the
	congestion of a physical edge $a$ is $m_a=|I(a)|$. Under an $(L,K)$-controlled
	scheme, $\sup_e|\gamma_e|\le L$ and $\sup_a m_a\le K$, and Theorem~\ref{thm:main}
	gives $\mathcal{E}^\Phi_H(\Theta) \le L\,\beta_\Phi(K)\, \mathcal{E}^\Phi_{\mathcal{B}^{\sharp}}(\theta)$.}

	\label{fig:fig01}
\end{figure}

\begin{theorem}[Gauge-energy repair]\label{thm:main}
	Suppose that $\mathcal{B}^{\sharp} = (T, F)$ is repaired in $H$ by an $(L,K)$-controlled
	repair scheme. Let $\Phi$ be a gauge such that $\alpha_\Phi(m) < \infty$ for
	$1 \le m \le K$. If $\mathcal{B}^{\sharp}$ supports a unit flow $\theta$ from $o$ to infinity with finite
	$\Phi$-energy, then the component of $\iota(o)$ in $H$ supports a unit flow $\Theta$
	satisfying
	\begin{equation}\label{eq:thm31}
	\mathcal{E}^\Phi_H(\Theta) \le L\,\beta_\Phi(K)\, \mathcal{E}^\Phi_{\mathcal{B}^{\sharp}}(\theta).
	\end{equation}
	In particular, for $\Phi(t) = t^q$, $q > 1$,
	\begin{equation}\label{eq:thm31q}
		\mathcal{E}^q_H(\Theta) \le L K^{q-1} \mathcal{E}^q_{\mathcal{B}^{\sharp}}(\theta).
	\end{equation}
\end{theorem}

\begin{proof}
	Orient every demand edge in \( F \). For each \( e \in F \), orient the repair path \( \gamma_e \) consistently with \( e \). That is, if  
	\( e = \{u, v\} \) is oriented from \( u \) to \( v \), then \( \gamma_e \) is oriented from \( \iota(u) \) to \( \iota(v) \). Let \( \eta_e \) denote the unit path  
	flow along \( \gamma_e \), defined as the antisymmetric edge function with \( \eta_e(a) = 1 \) if \( a \in \gamma_e \) is traversed in the  
	direction of \( \gamma_e \), \( \eta_e(a) = -1 \) if it is traversed in the opposite direction, and \( \eta_e(a) = 0 \) otherwise. Define
	\begin{equation}\label{eq:lift}
		\Theta = \sum_{e \in F} \theta(e)\, \eta_e.
	\end{equation}
	The sum is finite on each physical edge because $m_a \le K < \infty$, so $\Theta$ is a
	well-defined antisymmetric edge function on $H$.
	
	To identify its divergence, note that divergence is linear, so for every vertex $y$ of $H$,
	\begin{equation*}
		\divop \Theta(y) = \sum_{e \in F} \theta(e)\, \divop \eta_e(y).
	\end{equation*}
	If $y = \iota(x)$ for $x \in T$, the only nonzero contributions come from demand edges incident
	with $x$: an edge oriented out of $x$ contributes $+\theta(e)$ at $\iota(x)$, and an edge
	oriented into $x$ contributes $-\theta(e)$. Hence $\divop \Theta(\iota(x)) = \divop \theta(x)$.
	If $y \notin \iota(T)$, then $y$ is an internal vertex of every route $\gamma_e$ through it, at
	which the unit path flow has divergence zero; hence $\divop \Theta(y) = 0$. Since $\theta$ is a
	unit flow from $o$ to infinity, we obtain
	\begin{equation*}
		\divop \Theta(\iota(o)) = 1, \qquad \divop \Theta(y) = 0 \quad (y \ne \iota(o)).
	\end{equation*}
	Moreover, every demand edge in the support of $\theta$ lies in the component of $o$ in $\mathcal{B}^{\sharp}$,
	and every repair route connects physical terminals of that component; therefore $\Theta$ is
	supported on the component of $\iota(o)$ in $H$. Thus $\Theta$ is a unit flow from $\iota(o)$
	to infinity on that component.
	
	Fix $a \in E_H$ and set $m = m_a$. If $m = 0$, then $\Theta(a) = 0$. If $m \ge 1$, write
	$x_e = |\theta(e)|$ for $e \in I(a)$ and
	\begin{equation*}
		\bar{x} = \frac{1}{m} \sum_{e \in I(a)} x_e .
	\end{equation*}
	Then
	\begin{equation}\label{eq:merge}
		|\Theta(a)| \le \sum_{e \in I(a)} |\theta(e)| = m \bar{x}.
	\end{equation}
	By the estimate \eqref{eq:jensen}, we have
	\begin{equation}\label{eq:edge}	
		\begin{aligned}
			\Phi(|\Theta(a)|) \le \Phi(m \bar{x}) \le \alpha_\Phi(m) \Phi(\bar{x})
			\le \beta_\Phi(m_a) \sum_{e \in I(a)} \Phi(x_e)\\
			= \beta_\Phi(m_a) \sum_{e \in I(a)}\Phi(|\theta(e)|)
		\end{aligned}
	\end{equation}
	
	Since $m_a \le K$, the definition of $\beta_\Phi(K)$ gives $\beta_\Phi(m_a) \le \beta_\Phi(K)$,
	and therefore
	\begin{equation*}
		\Phi(|\Theta(a)|) \le \beta_\Phi(K) \sum_{e \in I(a)} \Phi(|\theta(e)|).
	\end{equation*}
	Summing over $a \in E_H$ and exchanging the order of summation,
	\begin{align}
		\mathcal{E}^\Phi_H(\Theta)
		&= \sum_{a \in E_H} \Phi(|\Theta(a)|)
		\le \beta_\Phi(K) \sum_{a \in E_H} \sum_{e \in I(a)} \Phi(|\theta(e)|) \nonumber\\
		&= \beta_\Phi(K) \sum_{e \in F} |\gamma_e|\, \Phi(|\theta(e)|)
		\le L\, \beta_\Phi(K) \sum_{e \in F} \Phi(|\theta(e)|)
		= L\, \beta_\Phi(K)\, \mathcal{E}^\Phi_{\mathcal{B}^{\sharp}}(\theta) < \infty, \label{eq:sum}
	\end{align}
	which proves \eqref{eq:thm31}. The power-gauge case \eqref{eq:thm31q} follows from
	\eqref{eq:power}, which gives $\beta_\Phi(K) = K^{q-1}$.
\end{proof}

	Moreover, this proof has a useful generalization. The proof of Theorem~\ref{thm:main} uses only the linearity of the divergence, the bound
$m_a \le K$, and the convexity estimate \eqref{eq:jensen}. It does not use the structure of
the flow beyond its finite $\Phi$-energy. The same argument therefore applies to finite-energy
flows between two finite terminal sets.  If $\theta$ is a unit flow from $u \in T$ to $v \in T$
in $\mathcal{B}^{\sharp}$ with finite $\Phi$-energy, then the lifted flow $\Theta$ is a unit
flow from $\iota(u)$ to $\iota(v)$ in $H$ and satisfies the same inequality
\eqref{eq:thm31}. 
We will give an explanation of the necessity of the parameters $K$ and $L$ under this condition in Proposition~\ref{prop:sharp} below.

\begin{remark}\label{rem:thm31}
	The constant $L\,\beta_\Phi(K)$ is the full cost of the repair scheme: the route length $L$
	charges every unit of macroscopic flow once per physical edge it traverses, and the dilation
	factor $\beta_\Phi(K)$ charges the worst-case congestion $K$. When the selected routes are
	edge-disjoint, $K = 1$ and $\beta_\Phi(1) = 1$, so \eqref{eq:thm31} reduces to
	$\mathcal{E}^\Phi_H(\Theta) \le L\, \mathcal{E}^\Phi_{\mathcal{B}^{\sharp}}(\theta)$: congestion is free only for disjoint
	routes, which is the setting of classical parallel-redundancy constructions. For the quadratic
	gauge, \eqref{eq:thm31q} reads $\mathcal{E}_H(\Theta) \le LK\, \mathcal{E}_{\mathcal{B}^{\sharp}}(\theta)$, which is the
	rough-embedding bound in energy form.
\end{remark}

When the gauge is quadratic, the gauge-energy estimate reduces to the classical finite-energy criterion for transience. This gives the following corollary.

\begin{corollary}[Quadratic transience]\label{cor:transience}
	Under the assumptions of Theorem~\ref{thm:main}, if  $\Phi(t) = t^2$ and $\mathcal{B}^{\sharp}$ is transient,
	then the repaired physical component is transient. Moreover,
	\begin{equation*}
		\mathcal{E}_H(\Theta) \le LK\, \mathcal{E}_{\mathcal{B}^{\sharp}}(\theta).
	\end{equation*}
\end{corollary}

\begin{proof}
	For $\Phi(t) = t^2$ we have $\alpha_\Phi(m) = m^2$, $\beta_\Phi(m) = m$ and
	$\beta_\Phi(K) = K$ by \eqref{eq:power}. Theorem~\ref{thm:main} gives a unit flow $\Theta$
	from $\iota(o)$ to infinity on the repaired component with
	\begin{equation*}
		\mathcal{E}_H(\Theta) = \sum_{a} \Theta(a)^2 \le LK \sum_{e} \theta(e)^2 = LK\, \mathcal{E}_{\mathcal{B}^{\sharp}}(\theta) < \infty.
	\end{equation*}
	By Thomson's principle, a locally finite graph is transient if and only if it admits a unit
	flow to infinity of finite quadratic energy~\cite{DS,LP}. Hence the component of $\iota(o)$ in $H$ is
	transient.
\end{proof}

The previous result assumes uniform repair costs through a common route-length bound and congestion
bound. We next extend this result to weighted settings where different repair paths and physical
edges may have different costs.
\begin{theorem}[Weighted variable-cost repair]\label{thm:weighted}
	Let $\mathcal{B}^{\sharp} = (T, F)$ be repaired in $H$ by paths $\{\gamma_e : e \in F\}$. Suppose physical
	edges carry deterministic weights $w_a > 0$. For each physical edge $a$, let $m_a = |I(a)|$,
	and define
	\begin{equation}\label{we:eq}
		W_e(\Phi, w) = \sum_{a \in \gamma_e} w_a\, \beta_\Phi(m_a).
	\end{equation}
	If $\mathcal{B}^{\sharp}$ supports a unit flow $\theta$ satisfying
	\begin{equation*}
		\sum_{e \in F} W_e(\Phi, w)\, \Phi(|\theta(e)|) < \infty,
	\end{equation*}
	then the lifted physical flow $\Theta$ satisfies
	\begin{equation*}
		\sum_{a \in E_H} w_a\, \Phi(|\Theta(a)|) \le \sum_{e \in F} W_e(\Phi, w)\, \Phi(|\theta(e)|).
	\end{equation*}
	For the unweighted case $w_a \equiv 1$, this reduces to the variable-cost estimate with
	\begin{equation*}
		W_e = \sum_{a \in \gamma_e} \beta_\Phi(m_a).
	\end{equation*}
	For quadratic weighted electrical networks, taking $\Phi(t) = t^2$ and $w_a = r_a$ gives
	\begin{equation}\label{eq:quadratic}
		\sum_{a \in E_H} r_a\, \Theta(a)^2 \le
		\sum_{e \in F} \Bigl( \sum_{a \in \gamma_e} r_a m_a \Bigr) \theta(e)^2 .
	\end{equation}
\end{theorem}

\begin{proof}
	The lifted flow is the same as in Theorem~\ref{thm:main}. For every physical edge $a$, the
	estimate \eqref{eq:edge} gives
	\begin{equation*}
		\Phi(|\Theta(a)|) \le \beta_\Phi(m_a) \sum_{e \in I(a)} \Phi(|\theta(e)|).
	\end{equation*}
	Multiplying by $w_a$ and summing over $a$ yields
	\begin{align*}
		\sum_{a \in E_H} w_a\, \Phi(|\Theta(a)|)
		&\le \sum_{a \in E_H} w_a\, \beta_\Phi(m_a) \sum_{e \in I(a)} \Phi(|\theta(e)|)\\
		&= \sum_{e \in F} \Bigl( \sum_{a \in \gamma_e} w_a\, \beta_\Phi(m_a) \Bigr)
		\Phi(|\theta(e)|),
	\end{align*}
	which is the desired inequality. The quadratic specialization \eqref{eq:quadratic} follows
	from $\beta_\Phi(m) = m$ for $\Phi(t) = t^2$.
\end{proof}

\begin{remark}\label{rem:weighted}
	The weighted formulation is the bridge to random repair costs: if
	the selected route for a demand edge has random length or passes through random
	edges with random weights, then $W_e(\Phi,w)$ in~\eqref{we:eq} is a random cost, and Theorem~\ref{thm:weighted}
	applies conditionally on the routes. This is how random repair lengths enter the
	argument in Section~\ref{sec:prob} and how chemical-distance estimates~\cite{AP,GaretMarchand2007,Luchtrath2026} are used in Section~\ref{sec5:6}. 
\end{remark}

We further establish the \textbf{sharpness of the lifting bound.}
The constant $L\,\beta_\Phi(K)$ in Theorem~\ref{thm:main} is not an artifact
of the proof. We show that for power gauges it is attained by an explicit
two-layer construction, and that neither factor can be dropped.

\begin{proposition}[Sharpness]\label{prop:sharp}
	Let $\Phi(t)=t^q$ with $q>1$. For every integers $L\ge 1$ and $K\ge 1$ there
	exist a bounded-degree demand network $\mathcal B$, a physical network $H$,
	an $(L,K)$-controlled repair scheme, and a unit flow $\theta$ from source to sink(  
	from $u$ to $v$ in $\mathcal B$) such that the lifted flow $\Theta$ satisfies
	\begin{equation}\label{eq:sharp}
		\mathcal{E}_H^\Phi(\Theta)=L\,\beta_\Phi(K)\,\mathcal{E}_{\mathcal B}^\Phi(\theta).
	\end{equation}
Consequently, within the source–sink formulation of Theorem~\ref{thm:main}, neither the route-length factor $L$ nor the dilation factor
	$\beta_\Phi(K)$ can be removed from~\eqref{eq:thm31}.
\end{proposition}

\begin{proof}
	Set $M=K$. Let $\mathcal B$ consist of $M$ parallel demand edges
	$e_1,\dots,e_M$ joining two terminals $u,v$. Define a unit source--sink
	flow by $\theta(e_i)=1/M$ for every $i$, so that
	$\divop\theta(u)=1$, $\divop\theta(v)=-1$, and $\divop\theta(x)=0$ otherwise.
	Then
	\[
	\mathcal{E}_{\mathcal B}^\Phi(\theta)=\sum_{i=1}^M \Phi(1/M)=M^{1-q}.
	\]
	Let $H$ contain a single path $\gamma=(a_1,\dots,a_L)$ of length $L$ joining
	$\iota(u)$ to $\iota(v)$, and route every $e_i$ through $\gamma$. On each
	physical edge $a_j$ all $M$ flows superpose in the same direction, so
	$\Theta(a_j)=M\cdot (1/M)=1$. Hence
	\[
	\mathcal{E}_H^\Phi(\Theta)=\sum_{j=1}^L \Phi(1)=L.
	\]
	The scheme has $\sup_e|\gamma_e|=L$ and $\sup_a m_a=M=K$, and for
	$\Phi(t)=t^q$ one has $\beta_\Phi(K)=K^{q-1}=M^{q-1}$. Therefore
	\[
	L\,\beta_\Phi(K)\, \mathcal{E}_{\mathcal B}^\Phi(\theta)
	=L\,M^{q-1}\cdot M^{1-q}=L
	=\mathcal{E}_H^\Phi(\Theta),
	\]
	which proves~\eqref{eq:sharp}.
\end{proof}

\begin{remark}[Isolating the two factors]\label{rem:sharp}
	The two extremal cases isolate the two contributions.
	\begin{itemize}
		\item \emph{Edge-disjoint routes ($K=1$).} Then $\beta_\Phi(1)=1$ and the
		construction gives $\mathcal{E}_H^\Phi(\Theta)=L\, \mathcal{E}_{\mathcal B}^\Phi(\theta)$.
		The factor $L$ is therefore unavoidable: a single demand edge routed over a
		path of length $L$ incurs $L$ times the energy of the edge itself.
		\item \emph{Single shared edge ($L=1$).} Then
		$\mathcal{E}_H^\Phi(\Theta)=\beta_\Phi(K)\,\mathcal{E}_{\mathcal B}^\Phi(\theta)$.
		The dilation factor $\beta_\Phi(K)=K^{q-1}$ is also unavoidable: $K$ parallel
		unit flows superposed on one edge carry $K^q$ energy against $K$ on the demand
		side.
	\end{itemize}
	For a general convex gauge $\Phi$, the same construction yields
	$\mathcal{E}_H^\Phi(\Theta)/\mathcal{E}_{\mathcal B}^\Phi(\theta)
	=L\,\Phi(M)/(M\,\Phi(1))$, which equals $L\,\beta_\Phi(M)$ whenever $t=1$
	attains the supremum in the definition of $\beta_\Phi(M)$. For power gauges
	this holds for every $t>0$, so equality is exact.
\end{remark}

\section{Sufficient conditions of probabilistic repair}\label{sec:prob}
In this section, we extend the deterministic repair theorems to random network failures. The key question is whether local repair routes remain available after Bernoulli edge percolation.
Let $H[p]$ denote Bernoulli bond percolation on the physical network. Large-scale properties of random subgraphs generated by percolation have been studied extensively, including geometric and analytic properties of infinite clusters~\cite{Barlow2004,Sapozhnikov2017}. For each demand edge $e$,
a family of candidate physical routes $\Gamma_e$ is specified. The demand edge is repairable
if at least one route in $\Gamma_e$ is open in $H[p]$. Its repairability indicator is denoted
by $Y_e$. We first describe the random repair model.

\begin{definition}[Macroscopic gauge-energy stability]\label{def:stability}
	Let $\mathcal{B}$ be an infinite bounded-degree demand network, and let $\Phi$ be a gauge. Fix
	$r \in (0, 1]$. We say that $(\mathcal{B}, \Phi, r)$ has \emph{macroscopic gauge-energy
		stability} if Bernoulli-$r$ bond percolation on $\mathcal{B}$ has, with positive probability,
	a connected component containing the root that supports a unit flow $\theta$ with
	\begin{equation*}
		\mathcal{E}^\Phi(\theta) < \infty.
	\end{equation*}
	For $\Phi(t) = t^2$, this is the usual transience of the root component.
\end{definition}

\begin{proposition}[Conditional repair criterion]\label{prop:conditional}
	On the event that a connected usable demand subnetwork $D = (T_D, F_D)$ containing $o$ supports
	a unit flow $\theta$ satisfying
	\begin{equation*}
		\sum_{e \in F_D} W_e(\Phi, w)\, \Phi(|\theta(e)|) < \infty,
	\end{equation*}
	where $W_e(\Phi, w)$ is the cost of the selected open repair route, the open physical cluster
	of $\iota(o)$ supports a finite weighted $\Phi$-energy unit flow. In the quadratic unweighted
	case, this cluster is transient.
\end{proposition}

\begin{proof}
	Condition on the percolation configuration and on the selected routes. On the event described,
	the selected open routes form a deterministic repair scheme for the subnetwork $D$ inside the
	open physical network $H[p]$, with route costs $W_e(\Phi, w)$. The demand flow $\theta$ has
	\begin{equation*}
		\sum_{e \in F_D} W_e(\Phi, w)\, \Phi(|\theta(e)|) < \infty,
	\end{equation*}
	so Theorem~\ref{thm:weighted} applies: the lifted flow $\Theta$ is a unit flow from
	$\iota(o)$ to infinity on the open cluster of $\iota(o)$ and satisfies
	\begin{equation*}
		\sum_{a} w_a\, \Phi(|\Theta(a)|)
		\le \sum_{e \in F_D} W_e(\Phi, w)\, \Phi(|\theta(e)|) < \infty.
	\end{equation*}
	Thus the open cluster supports a finite weighted $\Phi$-energy unit flow. In the quadratic
	unweighted case $w_a \equiv 1$ and $\Phi(t) = t^2$, this flow has finite quadratic energy, and
	the cluster is transient by Thomson's principle~\cite{DS,LP}.
\end{proof}

We use two standard inputs from percolation theory. First, the domination theorem of
Liggett, Schonmann and Stacey \cite{LSS} implies that, on the bounded-degree graph classes
under consideration, for every dependence range $R < \infty$ and every target density
$r \in (0, 1)$, there is a number $\rho = \rho(R, r) < 1$ such that every $R$-dependent bond
field $(Y_e)$ satisfying
\begin{equation}\label{eq:domination}
	\inf_{e} \mathbb{P}(Y_e = 1) \ge \rho
\end{equation}
stochastically dominates independent Bernoulli-$r$ bond percolation. 
Second, chemical-distance estimates for supercritical Bernoulli percolation on \(\mathbb{Z}^d\),
established by Antal and Pisztora \cite{AP} and further developed by
Garet and Marchand \cite{Garet2004}, provide linear control of open-path
lengths with exponentially decaying deviation probabilities.

The following theorem is the form in which the domination input is used in this paper.
Its proof is included because it is the bridge between finite-dependent local repair and the
deterministic gauge-energy comparison.

\begin{theorem}[Finite-dependent repair criterion]\label{thm:fd}
	Let $\mathcal{B}$ be an infinite bounded-degree demand network for which $(\mathcal{B},
	\Phi, r)$ has macroscopic gauge-energy stability. Let a bounded-range local repair model be
	placed over $\mathcal{B}$, and let $(Y_e)_{e \in E_{\mathcal{B}}}$ be its repairability field
	under Bernoulli failures of the physical edges. Assume that:
	\begin{enumerate}
		\item[(i)] $(Y_e)$ is $R$-dependent;
		\item[(ii)] for the domination threshold $\rho(R, r)$,
		\begin{equation*}
			\inf_{e \in E_{\mathcal{B}}} \mathbb{P}(Y_e = 1) \ge \rho(R, r);
		\end{equation*}
		\item[(iii)] whenever $Y_e = 1$, one can choose an open repair route for $e$ so that every
		selected route family over the subgraph under consideration is $(L,K)$-controlled.
	\end{enumerate}
	Then, with positive probability, the open physical network contains a component supporting a
	finite $\Phi$-energy unit flow. More precisely, on the event that the dominated Bernoulli-$r$
	demand component $D$ carries a unit flow $\theta$ with $\mathcal{E}^\Phi_D(\theta) < \infty$, the
	lifted physical flow $\Theta$ satisfies
	\begin{equation}\label{eq:thm41}
		\mathcal{E}^\Phi_H(\Theta) \le L\, \beta_\Phi(K)\, \mathcal{E}^\Phi_D(\theta).
	\end{equation}
	For $\Phi(t) = t^2$, this physical component is transient.
\end{theorem}

\begin{proof}
	By the finite-dependent domination input, there exists a coupling of the repairability field
	$(Y_e)$ with an independent Bernoulli-$r$ bond percolation field $(\omega_e)$ on
	$E_{\mathcal{B}}$ such that
	\begin{equation*}
		\omega_e \le Y_e \quad \text{for all } e \in E_{\mathcal{B}}
	\end{equation*}
	almost surely. Let
	\begin{equation*}
		F_D = \{e \in E_{\mathcal{B}} : \omega_e = 1\}
	\end{equation*}
	and let $D = (T_D, F_D)$ be the corresponding open demand subnetwork. Since
	$(\mathcal{B}, \Phi, r)$ has macroscopic gauge-energy stability, with positive probability the
	component of $o$ in $D$ supports a unit flow $\theta$ satisfying
	\begin{equation*}
		\mathcal{E}^\Phi_D(\theta) = \sum_{e \in F_D} \Phi(|\theta(e)|) < \infty.
	\end{equation*}
	On the same event, every edge $e \in F_D$ is repairable, because $\omega_e = 1$ implies
	$Y_e = 1$. Choose for every $e \in F_D$ one open repair route $\gamma_e$. By assumption
	(iii), the chosen routes may be selected so that
	\begin{equation*}
		\sup_{e \in F_D} |\gamma_e| \le L, \qquad
		\sup_{a \in E_H} \#\{e \in F_D : a \in \gamma_e\} \le K.
	\end{equation*}
	The selected routes form an $(L,K)$-controlled repair scheme for $D$ inside the open physical
	network, and the deterministic repair theorem therefore gives
	\begin{equation*}
		\mathcal{E}^\Phi_H(\Theta) \le L\, \beta_\Phi(K) \sum_{e \in F_D} \Phi(|\theta(e)|)
		= L\, \beta_\Phi(K)\, \mathcal{E}^\Phi_D(\theta) < \infty.
	\end{equation*}
	When $\Phi(t) = t^2$, the existence of a finite quadratic-energy unit flow is equivalent to
	transience by Thomson's principle.
\end{proof}

\begin{remark}\label{rem:fd}
	The theorem is a sufficient condition in the following precise sense: it requires only the marginal
	condition (ii), the finite dependence range in (i), and the deterministic control (iii) of the
	selected routes; it does not require independence of the repair mechanism, nor any
	information about the geometry of the physical network beyond (iii). The threshold
	$\rho(R, r)$ is generally not explicit, but for every fixed $R < \infty$ and $r \in (0,1)$ it
	is strictly below $1$, so condition (ii) is a high-density condition on the local repair
	mechanism.
\end{remark}

The next theorem records the random-cost form needed for block and coarse-graining
applications. It is a direct consequence of conditional integrability and the weighted
variable-cost theorem, but the proof is written out to make explicit where random repair
lengths enter the argument.

\begin{theorem}[Random-cost repair sufficient condition]\label{thm:random}
	Let $D = (T, F)$ be a random usable demand subnetwork with selected open repair routes and
	random costs $W_e \ge 0$. Suppose that, conditional on $D$, there is a unit flow $\theta$ on
	$D$ such that
	\begin{equation*}\label{the4:5:01}
		\sum_{e \in F} \mathbb{E}[W_e \mid D]\, \Phi(|\theta(e)|) < \infty.
	\end{equation*}
	Then
	\begin{equation}\label{the4:5:02}
		\sum_{e \in F} W_e\, \Phi(|\theta(e)|) < \infty \quad \text{almost surely},
	\end{equation}
	and the corresponding physical component supports a finite $\Phi$-energy lifted flow. In
	particular, if deterministic numbers $C_e$ satisfy
	\begin{equation*}
		\mathbb{E}[W_e \mid D] \le C_e, \qquad \sum_{e \in F} C_e\, \Phi(|\theta(e)|) < \infty,
	\end{equation*}
	then the conclusion~(\ref{the4:5:02}) holds.
\end{theorem}

\begin{proof}
	Set
	\begin{equation*}
		Z = \sum_{e \in F} W_e\, \Phi(|\theta(e)|).
	\end{equation*}
	All summands are nonnegative, hence Tonelli's theorem gives, conditionally on $D$,
	\begin{equation*}
		\mathbb{E}[Z \mid D] = \mathbb{E}\Bigl[ \sum_{e \in F} W_e\, \Phi(|\theta(e)|) \Bigm| D \Bigr]
		= \sum_{e \in F} \mathbb{E}[W_e \mid D]\, \Phi(|\theta(e)|) < \infty.
	\end{equation*}
	A nonnegative random variable with finite conditional expectation is finite almost surely;
	indeed $\mathbb{P}(Z = \infty \mid D) = 0$. Hence $Z < \infty$ almost surely. Condition on $D$
	and on the selected open routes. The selected routes define a deterministic variable-cost
	repair scheme to which Theorem~\ref{thm:weighted} applies with the realized costs $W_e$, so
	the lifted flow satisfies
	\begin{equation*}
		\mathcal{E}^\Phi_H(\Theta) \le \sum_{e \in F} W_e\, \Phi(|\theta(e)|) = Z < \infty.
	\end{equation*}
	If the deterministic upper bounds $C_e$ are available, then
	\begin{equation*}
		\sum_{e \in F} \mathbb{E}[W_e \mid D]\, \Phi(|\theta(e)|)
		\le \sum_{e \in F} C_e\, \Phi(|\theta(e)|) < \infty,
	\end{equation*}
	so the preceding argument applies. Chemical-distance estimates, such as those of Antal and
	Pisztora \cite{AP}, provide a standard way to verify such bounds in supercritical block
	constructions.
\end{proof}

For later reference we also recall the macroscopic transience input of Grimmett, Kesten and
Zhang \cite{GKZ}: if $d \ge 3$ and $r > p_c(\mathbb{Z}^d)$, then the infinite cluster of
Bernoulli-$r$ bond percolation on $\mathbb{Z}^d$ is transient almost surely on the event of its
existence. Equivalently, that cluster supports a finite quadratic-energy unit flow.

\section{Applications}\label{sec:app}

This section explains how the abstract repair theorems are used in concrete network models.
The point of the applications is not to introduce a new class of trees, but to show that the
repair results turns known macroscopic flow information into statements about locally
reinforced physical networks. We first record a general reinforcement principle, then give two
local repair mechanisms involving independent corridors and finite-dependent square bypasses, and
finally discuss lattice, wedge-type, coarse block and finite-scale resistance consequences.

\subsection{A general local-reinforcement principle}

The following statement is the basic way in which the paper should be used. A macroscopic
theorem supplies finite $\Phi$-energy flows on a random subnetwork of the demand graph. A
local repair mechanism then realizes the corresponding demand edges inside the physical graph. 

\begin{corollary}[Local reinforcement transfer]\label{cor:transfer}
	Let $\mathcal{B}$ be a bounded-degree demand network for which $(\mathcal{B}, \Phi, r)$ has
	macroscopic gauge-energy stability. Suppose that each demand edge is equipped with a
	bounded-range local repair module. Let $Y_e$ be the event that demand edge $e$ is repairable
	after Bernoulli failures in the physical network. Assume that:
	\begin{enumerate}
		\item[(i)] the field $(Y_e)$ is $R$-dependent;
		\item[(ii)] $\inf_e \mathbb{P}(Y_e = 1) \ge \rho(R, r)$, where $\rho(R, r)$ is the domination
		threshold in Theorem~\ref{thm:fd};
		\item[(iii)] selected open repair routes can be chosen with length at most $L$ and congestion
		at most $K$.
	\end{enumerate}
	Then the surviving physical network supports a finite $\Phi$-energy unit flow with positive
	probability. Moreover, if $\theta$ is the macroscopic flow on the dominated Bernoulli-$r$
	demand component and $\Theta$ is its lift, then
	\begin{equation*}
		\mathcal{E}^\Phi_H(\Theta) \le L\, \beta_\Phi(K)\, \mathcal{E}^\Phi_D(\theta).
	\end{equation*}
\end{corollary}

\begin{proof}
	The assumptions are exactly those of Theorem~\ref{thm:fd}. The domination step produces a
	coupling under which the repairability field $(Y_e)$ dominates an independent Bernoulli-$r$
	field $(\omega_e)$ with $\omega_e \le Y_e$ for all $e$, and $D = (T, F_D)$ with
	$F_D = \{e : \omega_e = 1\}$ is the dominated usable demand subnetwork. Macroscopic
	gauge-energy stability gives, with positive probability, a unit flow $\theta$ on the component
	of $o$ in $D$ with $\mathcal{E}^\Phi_D(\theta) < \infty$. On that event, every edge of $F_D$ is
	repairable, and by (iii) the selected open repair routes form an $(L,K)$-controlled repair
	scheme for $D$ in the open physical network. Theorem~\ref{thm:main} gives the displayed bound.
\end{proof}

This corollary is useful because all model-specific work is pushed into three checkable
quantities: the repair probability, the dependence range and the congestion of selected repair
routes. The next two subsections make these quantities explicit in representative local
modules.

\subsection{Independent corridor reinforcement}

The simplest repair module replaces each demand edge by several short independent physical
corridors. This case is useful as a benchmark because it gives an explicit reliability
threshold and does not require stochastic domination for dependent fields.

Fix a demand network $\mathcal{B} = (T, E_{\mathcal{B}})$. For each demand edge
$e = \{u, v\}$, add $M$ internally edge-disjoint physical paths of length $\ell$ between
$\iota(u)$ and $\iota(v)$. Assume that the physical edges used for different demand edges are
disjoint. Under Bernoulli-$p$ failures, demand edge $e$ is repairable if at least one of the
$M$ length-$\ell$ corridors is open. Hence
\begin{equation}\label{eq:corridor}
	q_{M,\ell}(p) = \mathbb{P}(Y_e = 1) = 1 - (1 - p^{\ell})^{M}.
\end{equation}
The indicators $(Y_e)$ are independent, and the usable demand network is exactly
Bernoulli-$q_{M,\ell}(p)$ percolation on $\mathcal{B}$.

\begin{proposition}[Independent corridor transfer]\label{prop:corridor}
	Let $\mathcal{B}$ be a demand network for which $(\mathcal{B}, \Phi, r)$ has macroscopic
	gauge-energy stability. In the independent corridor reinforcement above, if
	\begin{equation*}
		q_{M,\ell}(p) = 1 - (1 - p^{\ell})^{M} \ge r,
	\end{equation*}
	then the surviving physical network supports a finite $\Phi$-energy unit flow with positive
	probability. If one open corridor is selected for every usable demand edge, then the lifted
	flow satisfies
	\begin{equation*}
		\mathcal{E}^\Phi_H(\Theta) \le \ell\, \beta_\Phi(K_0)\, \mathcal{E}^\Phi_D(\theta),
	\end{equation*}
	where $K_0$ is the deterministic corridor overlap bound. In the edge-disjoint corridor
	construction $K_0 = 1$, so
	\begin{equation*}
		\mathcal{E}^\Phi_H(\Theta) \le \ell\, \mathcal{E}^\Phi_D(\theta).
	\end{equation*}
	Equivalently, for a target macroscopic parameter $r$ it is enough that
	\begin{equation}\label{eq:threshold}
		p \ge \bigl[ 1 - (1 - r)^{1/M} \bigr]^{1/\ell}.
	\end{equation}
\end{proposition}

\begin{proof}
	For a fixed demand edge $e$, a given corridor is open with probability $p^{\ell}$, and the $M$
	corridors are internally edge-disjoint, hence independently open. The event that all $M$
	corridors fail is therefore $(1 - p^{\ell})^{M}$, which gives \eqref{eq:corridor}. Since the
	physical edges used for different demand edges are disjoint, the indicators $(Y_e)$ are
	independent, and the usable demand network is exactly Bernoulli-$q_{M,\ell}(p)$ percolation on
	$\mathcal{B}$; in particular it stochastically dominates Bernoulli-$r$ percolation whenever
	$q_{M,\ell}(p) \ge r$. On the event that the dominated Bernoulli-$r$ demand component carries
	a finite $\Phi$-energy flow $\theta$, choose for each edge used by $\theta$ one open corridor.
	The route length is $\ell$ and the overlap is bounded by $K_0$, so Theorem~\ref{thm:main}
	gives
	\begin{equation*}
		\mathcal{E}^\Phi_H(\Theta) \le \ell\, \beta_\Phi(K_0)\, \mathcal{E}^\Phi_D(\theta).
	\end{equation*}
	Solving $1 - (1 - p^{\ell})^{M} \ge r$ for $p$ gives \eqref{eq:threshold}.
\end{proof}

This proposition gives a direct reliability interpretation. Increasing the number of local
corridors lowers the microscopic failure tolerance needed to realize the same macroscopic
percolation parameter. For fixed $r$ and $\ell$,
\begin{equation*}
	\bigl[ 1 - (1 - r)^{1/M} \bigr]^{1/\ell} \asymp \Bigl( \frac{-\log(1 - r)}{M} \Bigr)^{1/\ell}
	\qquad (M \to \infty),
\end{equation*}
since $(1 - r)^{1/M} = \exp(M^{-1} \log(1 - r)) = 1 - M^{-1}|\log(1 - r)| + O(M^{-2})$. Thus
local parallel redundancy improves the admissible microscopic threshold at polynomial order
$M^{-1/\ell}$.

\subsection{A finite-dependent square-bypass module}

Independent corridors are mathematically transparent, but they are too idealized for many
repair networks. Local bypasses usually share physical edges with nearby modules. The
square-bypass construction below is a concrete finite-dependent model showing why
Theorem~\ref{thm:fd} is needed.

Let the demand network be $\mathcal{B} = \mathbb{Z}^d$ with nearest-neighbour edges. Embed the
terminals in a physical copy of $\mathbb{Z}^d$. For a demand edge $e = \{x, x + e_i\}$, allow the
following candidate routes:
\begin{enumerate}
	\item[(a)] the direct edge $\{x, x + e_i\}$;
	\item[(b)] for each $j \ne i$ and $\sigma \in \{-1, 1\}$, the length-three square bypass
	\begin{equation*}
		x \to x + \sigma e_j \to x + e_i + \sigma e_j \to x + e_i .
	\end{equation*}
\end{enumerate}

There are $2(d-1)$ square bypasses. For a fixed demand edge, the direct edge and these bypasses
are edge-disjoint. Therefore
\begin{equation}\label{eq:sq}
	\pi_{\mathrm{sq}}(p, d) = 1 - (1 - p)(1 - p^3)^{2(d-1)}.
\end{equation}

The field $(Y_e)$ is not independent, since nearby demand edges may use common physical edges.
It is, however, finite-dependent. Indeed, $Y_e$ is a function only of physical edges in the
finite union of unit squares adjacent to $e$. Hence there exists a deterministic number
$R_d < \infty$ such that $Y_e$ and $Y_f$ are independent whenever the edge-distance between
$e$ and $f$ in the demand lattice exceeds $R_d$.

We next verify the route-length and congestion bounds.
Define
\begin{equation*}
	K_d = \sup_{a \in E_H} \#\{e \in E_{\mathcal{B}} : a \text{ belongs to at least one candidate
		route for } e\}.
\end{equation*}
Since every candidate route remains in a bounded neighbourhood of its demand edge,
$K_d < \infty$. Any selection of one route for each repairable demand edge then has congestion
at most $K_d$ and length at most $3$.

The preceding observations verify the hypotheses of Theorem~\ref{thm:fd}. We therefore obtain the following proposition.

\begin{proposition}[Square-bypass lattice criterion]\label{prop:sq}
	Let $d \ge 3$, let $\mathcal{B} = \mathbb{Z}^d$, and consider the square-bypass repair model under
	Bernoulli-$p$ physical edge failures. Fix a gauge $\Phi$ and a parameter $r$ such that
	$(\mathbb{Z}^d, \Phi, r)$ has macroscopic gauge-energy stability. If
	\begin{equation}\label{eq:sqcond}
		\pi_{\mathrm{sq}}(p, d) = 1 - (1 - p)(1 - p^3)^{2(d-1)} \ge \rho(R_d, r),
	\end{equation}
	then the surviving physical network supports a finite $\Phi$-energy unit flow with positive
	probability. The lifted flow may be chosen to satisfy
	\begin{equation*}
		\mathcal{E}^\Phi_H(\Theta) \le 3\, \beta_\Phi(K_d)\, \mathcal{E}^\Phi_D(\theta).
	\end{equation*}
	In particular, for $\Phi(t) = t^2$, the surviving physical network contains a transient
	component with positive probability whenever $r > p_c(\mathbb{Z}^d)$ and \eqref{eq:sqcond} holds.
\end{proposition}

\begin{proof}
	For a fixed demand edge $e$, the direct edge is open with probability $p$ and each of the
	$2(d-1)$ square bypasses is open with probability $p^3$; these $1 + 2(d-1)$ routes are pairwise
	edge-disjoint, so the repair probability is
	\begin{equation*}
		\mathbb{P}(Y_e = 1) = 1 - (1 - p)(1 - p^3)^{2(d-1)} = \pi_{\mathrm{sq}}(p, d),
	\end{equation*}
	which is \eqref{eq:sq}. The field $(Y_e)$ is $R_d$-dependent by the argument above, and the
	marginal condition \eqref{eq:sqcond} permits the application of the Liggett--Schonmann--Stacey(LSS) domination theorem~\cite{LSS}
	with target parameter $r$: there is a coupling of $(Y_e)$ with an independent Bernoulli-$r$
	field $(\omega_e)$ on $E(\mathbb{Z}^d)$ under which
	\begin{equation*}
		\omega_e \le Y_e \quad (e \in E(\mathbb{Z}^d)).
	\end{equation*}
	By macroscopic gauge-energy stability, with positive probability the open demand component of
	$(\omega_e)$ contains a unit flow $\theta$ with finite $\Phi$-energy. Since every open
	$\omega$-edge is repairable, choose one open direct or bypass route for each edge used by
	$\theta$. The route length is at most $3$ and the congestion is at most $K_d$.
	Theorem~\ref{thm:main} yields the displayed inequality. For $\Phi(t) = t^2$, macroscopic
	stability follows from the transience theorem of Grimmett, Kesten and Zhang \cite{GKZ} for
	supercritical percolation on $\mathbb{Z}^d$, $d \ge 3$.
\end{proof}

\begin{remark}\label{rem:sq}
	The condition \eqref{eq:sqcond} is implicit because the LSS threshold is generally not
	explicit. The local part is explicit: as $p \uparrow 1$,
	\begin{equation*}
		\pi_{\mathrm{sq}}(p, d) \uparrow 1,
	\end{equation*}
	so for every fixed $r \in (0,1)$ and dependence range $R_d$, the condition holds for all $p$
	sufficiently close to one.
\end{remark}

\subsection{Gauge-energy inputs on lattices}

The preceding propositions require macroscopic gauge-energy stability. For lattices this
requirement is natural and can be seen already at the deterministic level. For the power gauge
$\Phi(t) = t^q$, a radial unit flow on $\mathbb{Z}^d$ has finite $q$-energy precisely in the classical
range
\begin{equation}\label{eq:radial}
	q > \frac{d}{d - 1}.
\end{equation}
Indeed, send approximately equal current through the edges crossing the sphere of radius $n$:
the $n$-th shell contains order $n^{d-1}$ edges, and a typical edge crossing that shell
carries current of order $n^{1-d}$. The contribution of the $n$-th shell to the $q$-energy is
therefore of order
\begin{equation*}
	n^{d-1} \bigl( n^{1-d} \bigr)^{q} = n^{(d-1)(1-q)},
\end{equation*}
and the series
\begin{equation*}
	\sum_{n \ge 1} n^{(d-1)(1-q)}
\end{equation*}
converges exactly when $(d-1)(q-1) > 1$, that is, $q > d/(d-1)$. Hoffman--Mossel type results
\cite{HM} show that such finite gauge-energy inputs persist on supercritical percolation
clusters of $\mathbb{Z}^d$ in the corresponding setting. In fact, by the sharp results of Levin and
Peres \cite{LP2}, for $d \ge 3$ the supercritical cluster of $\mathbb{Z}^d$ supports a nonzero flow
of finite $q$-energy precisely in the same range \eqref{eq:radial}, so the deterministic
threshold is not an artifact of the radial construction. In two dimensions the deterministic
threshold is $q > 2$, and this is also the sharp percolation threshold: Hoffman \cite{Hoffman}
showed that for $p > p_c(\mathbb{Z}^2)$ the infinite cluster of $\mathbb{Z}^2$ supports a nonzero flow of
finite $q$-energy for every $q > 2$, and no such flow exists for $q \le 2$ since $\mathbb{Z}^2$
itself admits no nonzero flow of finite $q$-energy in that range.

\begin{corollary}[Reinforced lattice gauge-energy criterion]\label{cor:lattice}
	Let $d \ge 3$ and $q > d/(d-1)$. Suppose that Bernoulli-$r$ percolation on $\mathbb{Z}^d$ has, with
	positive probability, a root component carrying a finite $q$-energy unit flow. In the
	square-bypass repair model, if
	\begin{equation*}
		1 - (1 - p)(1 - p^3)^{2(d-1)} \ge \rho(R_d, r),
	\end{equation*}
	then the surviving locally repaired physical network supports a finite $q$-energy unit flow
	with positive probability. For the lifted flow,
	\begin{equation*}
		\mathcal{E}^q_H(\Theta) \le 3\, K_d^{q-1}\, \mathcal{E}^q_D(\theta).
	\end{equation*}
\end{corollary}

\begin{proof}
	This is Proposition~\ref{prop:sq} with $\Phi(t) = t^q$. For this gauge,
	\eqref{eq:power} gives $\beta_\Phi(K_d) = K_d^{q-1}$, and the bound of
	Proposition~\ref{prop:sq} becomes
	\begin{equation*}
		\mathcal{E}^q_H(\Theta) \le 3\, \beta_\Phi(K_d)\, \mathcal{E}^q_D(\theta)
		= 3\, K_d^{q-1}\, \mathcal{E}^q_D(\theta).
	\end{equation*}
\end{proof}

\subsection{Wedge-type demand networks}

Wedge-type networks are important because their transience and gauge-energy behaviour is not
merely a restatement of the lattice case. The repair framework can be applied to any wedge
demand network for which an external theorem gives a macroscopic finite-energy input. This
includes the transience and convex-gauge results proved for suitable wedges by Angel,
Benjamini, Berger and Peres \cite{ABBP}.
Combining these wedge results with the repair framework gives the following corollary.

\begin{corollary}[Locally reinforced wedge-type networks]\label{cor:wedge}
	Let $W$ be a wedge-type demand network satisfying macroscopic gauge-energy stability for
	$(\Phi, r)$. Suppose a bounded-range local repair model over $W$ has repairability field
	$(Y_e)$ that is $R$-dependent, has
	\begin{equation*}
		\inf_{e} \mathbb{P}(Y_e = 1) \ge \rho(R, r),
	\end{equation*}
	and admits selected open routes of length at most $L$ and congestion at most $K$. Then the
	surviving physical network supports a finite $\Phi$-energy unit flow with positive
	probability. The lifted flow satisfies
	\begin{equation*}
		\mathcal{E}^\Phi_H(\Theta) \le L\, \beta_\Phi(K)\, \mathcal{E}^\Phi_D(\theta).
	\end{equation*}
\end{corollary}

\begin{proof}
	By the finite-dependent domination input, the repairability field $(Y_e)$ dominates an
	independent Bernoulli-$r$ field $(\omega_e)$ on the edge set of $W$. On the event supplied by
	macroscopic gauge-energy stability, the component of $o$ in the dominated demand subnetwork
	$D$ carries a unit flow $\theta$ with $\mathcal{E}^\Phi_D(\theta) < \infty$. Since $\omega_e \le Y_e$,
	every demand edge used by $\theta$ is repairable; choosing open repair routes of length at
	most $L$ and congestion at most $K$, and applying Theorem~\ref{thm:main}, gives the stated
	bound.
\end{proof}

This corollary is intentionally phrased as a transfer theorem. The wedge theorem supplies the
macroscopic input; the present paper supplies the local repair step. In this way a
gauge-energy result on a wedge can be converted into a gauge-energy result for a locally
reinforced physical wedge network without reproving the wedge theorem itself.

\subsection{Coarse block repairs and random route costs}\label{sec5:6}

Bounded local repairs are not the only natural model. In block constructions, a demand edge
may be realized by an open path through a random supercritical region, and the repair length
is random. The variable-cost theorem was introduced precisely for this situation.

Consider a coarse demand network whose vertices correspond to large blocks in a physical
percolation configuration. A coarse demand edge is declared usable if the corresponding
neighbouring blocks are connected by an open path inside a prescribed enlarged box. For
supercritical Bernoulli percolation on $\mathbb{Z}^d$, chemical-distance estimates imply that such
connecting paths have length of order the block size with exponentially small upper-tail
probability, conditional on the relevant blocks being good. The chemical-distance bounds of Antal and Pisztora \cite{AP}, in the large-deviation
form of Garet and Marchand \cite{GaretMarchand2007}, give
\[
\mathbb{E}[W_e \mid D] \leq C_N,
\]
where \(C_N\) depends on the block scale and the percolation parameter. These
bounds imply that, for all large block scales, the connecting paths between
adjacent blocks have length of order the block size with exponentially small
upper-tail probability. Consequently, the conditional cost bound holds with
\(C_N\) of order \(N\). We treat this bound as a hypothesis of the repair
theorem below, rather than proving it as a separate lemma.

\begin{proposition}[Fixed-scale block repair]\label{prop:block}
	Let $D = (T, F)$ be a usable coarse demand network obtained from a block construction.
	Suppose each selected open repair route has random cost $W_e$ and that, conditional on $D$,
	there is a unit flow $\theta$ on $D$ such that
	\begin{equation*}
		\sum_{e \in F} \mathbb{E}[W_e \mid D]\, \Phi(|\theta(e)|) < \infty.
	\end{equation*}
	Then the surviving physical network supports a finite $\Phi$-energy lifted flow. In
	particular, if for a fixed block scale $N$,
	\begin{equation*}
		\sup_{e \in F} \mathbb{E}[W_e \mid D] \le C_N < \infty
	\end{equation*}
	and $\mathcal{E}^\Phi_D(\theta) < \infty$, then the conclusion holds.
\end{proposition}

\begin{proof}
	The first conclusion is Theorem~\ref{thm:random}. Under the uniform fixed-scale bound, for
	every $e \in F$ we have $\mathbb{E}[W_e \mid D] \le C_N$, hence
	\begin{equation*}
		\sum_{e \in F} \mathbb{E}[W_e \mid D]\, \Phi(|\theta(e)|)
		\le C_N \sum_{e \in F} \Phi(|\theta(e)|)
		= C_N\, \mathcal{E}^\Phi_D(\theta) < \infty,
	\end{equation*}
	so the hypothesis of Theorem~\ref{thm:random} is satisfied.
\end{proof}

This proposition separates the role of block percolation from the role of energy comparison.
Renormalization and chemical-distance estimates produce good coarse edges and control repair
costs; the repair theorem then lifts the macroscopic flow.

\subsection{Finite-scale effective resistance and capacity comparison}

The quadratic case also gives a finite-scale comparison. This is useful for potential-theoretic
applications because transience is an asymptotic statement, while effective resistance and
capacity can be compared at finite boundaries.

Let $A \subseteq T$ be a finite terminal boundary. Suppose that a finite demand subnetwork
connecting $o$ to $A$ is repaired in $H$ by routes of length at most $L$ and congestion at
most $K$. Let $R^{\mathrm{eff}}$ denote efficient resistance in the graph. Then
\begin{equation}\label{eq:resistance}
	R^{\mathrm{eff}}_H(\iota(o) \leftrightarrow \iota(A))
	\le LK\, R^{\mathrm{eff}}_{\mathcal{B}}(o \leftrightarrow A).
\end{equation}
Indeed, take an arbitrary unit flow $\theta$ from $o$ to $A$ in the demand network.
Corollary~\ref{cor:transience} gives a lifted unit flow $\Theta$ from $\iota(o)$ to
$\iota(A)$ satisfying
\begin{equation*}
	\mathcal{E}_H(\Theta) \le LK\, \mathcal{E}_{\mathcal{B}}(\theta).
\end{equation*}
By Thomson's principle, $R^{\mathrm{eff}}_H(\iota(o) \leftrightarrow \iota(A))$ is the
infimum of $\mathcal{E}_H$ over unit flows from $\iota(o)$ to $\iota(A)$, so
\begin{equation*}
	R^{\mathrm{eff}}_H(\iota(o) \leftrightarrow \iota(A))
	\le \mathcal{E}_H(\Theta) \le LK\, \mathcal{E}_{\mathcal{B}}(\theta)
\end{equation*}
for every demand unit flow $\theta$; taking the infimum over $\theta$ yields
\eqref{eq:resistance}. Equivalently, in terms of capacity,
\begin{equation*}
	\mathrm{Cap}_H(\iota(o), \iota(A)) \ge \frac{1}{LK}\, \mathrm{Cap}_{\mathcal{B}}(o, A).
\end{equation*}
Thus the repair framework gives resistance-profile bounds and not only qualitative transience.

\bibliographystyle{model2-names}
\bibliography{myreferences}

\end{document}